\documentclass[11pt,a4paper]{amsart}
\usepackage[utf8]{inputenc}
\usepackage{amsmath}
\usepackage{amsfonts}
\usepackage{amssymb}
\usepackage{amsthm}
\usepackage{graphicx}
\usepackage[left=2.5cm,right=2.5cm,top=3cm,bottom=3cm]{geometry}
\usepackage{url}

\usepackage{mathtools}
\mathtoolsset{showonlyrefs=true}
\DeclarePairedDelimiter\paren{\lparen}{\rparen}

\usepackage[linesnumbered,ruled]{algorithm2e}

\newcommand{\fl}[1]{{(-\triangle)^{#1}}}
\newcommand{\dfl}[1]{{(-\triangle)_{\rmd}^{#1}}}

\newcommand{\calF}{\mathcal{F}}
\newcommand{\calH}{\mathcal{H}}
\newcommand{\calM}{\mathcal{M}}
\newcommand{\bbC}{\mathbb{C}}

\newcommand{\bbR}{\mathbb{R}}
\newcommand{\bbT}{\mathbb{T}}
\newcommand{\bbZ}{\mathbb{Z}}

\newcommand{\rme}{\mathrm{e}}
\newcommand{\rmd}{\mathrm{d}}
\newcommand{\rmO}{\mathrm{O}}
\newcommand{\rmi}{\mathrm{i}}
\newcommand{\olu}{\overline{u}}

\newcommand{\Md}{\mathcal{M}_\rmd}
\newcommand{\Hd}{\mathcal{H}_\rmd}

\usepackage{bm}
\newcommand{\bU}{\bm U}
\newcommand{\bV}{\bm V}
\newcommand{\bb}{\bm b}
\newcommand{\br}{\bm r}
\newcommand{\bp}{\bm p}
\newcommand{\bs}{\bm s}
\newcommand{\bt}{\bm t}
\newcommand{\bx}{\bm x}
\newcommand{\by}{\bm y}
\newcommand{\bv}{\bm v}

\DeclareMathOperator*{\mathspan}{span}
\DeclareMathOperator*{\diag}{diag}
\DeclareMathOperator*{\sech}{sech}

\theoremstyle{definition}

\newtheorem{remark}{Remark}
\newtheorem{theorem}{Theorem}
\newtheorem{lemma}{Lemma}
\newtheorem{proposition}{Proposition}

\usepackage{tikz}
\usepackage{pgfplotstable}
\usetikzlibrary{arrows,positioning,plotmarks,external,patterns,angles,
decorations.pathmorphing,backgrounds,fit,shapes,graphs,calc}
\usepgfplotslibrary{external}
\author{Yuto~Miyatake, Tai~Nakagawa, Tomohiro~Sogabe, Shao-Liang~Zhang}
\title[A Linearly Implicit scheme for the fractional nonlinear Schr\"odinger equation]{A structure-preserving Fourier pseudo-spectral linearly implicit scheme for the space-fractional nonlinear Schr\"odinger equation}

\allowdisplaybreaks

\begin{document}
\maketitle

\begin{abstract}
We propose a Fourier pseudo-spectral scheme
  for the space-fractional nonlinear Schr\"odinger equation.
  The proposed scheme has the following features:
  it is linearly implicit,
  it preserves two invariants of the equation,
  its unique solvability is guaranteed without
  any restrictions on space and time step sizes.
  The scheme requires solving a complex symmetric linear system per time step. To solve the system efficiently, we also present a certain variable transformation and preconditioner.
\end{abstract}

\section{Introduction}
\label{sec1}

The Schr\"odinger equation, a fundamental equation in quantum mechanics,
can be derived by
using Feynman path integrals over Brownian trajectories~\cite{fh65}.
Around 2000,
Laskin coined the fractional Schr\"odinger equation by
replacing the Feynman path integrals by the L\'evy ones~\cite{la00b,la02,la00}.
For example, the
space-fractional nonlinear Schr\"odinger (FNLS) equation with the cubic
nonlinearity is formulated as
\begin{align} 
u_t =-\rmi \fl{\frac{\alpha}{2}} u + \rmi |u|^2 u, \quad  t>0 ,
\end{align}
where $\rmi = \sqrt{-1}$, the subscript $t$ denotes the differentiation with respect to time variable $t$, and
$\fl{\frac{\alpha}{2}}$ denotes the fractional Laplacian
with the L\'evy index $1<\alpha \leq 2$.
In one-dimensional cases, the fractional Laplacian is defined by
\begin{align}
\fl{\frac{\alpha}{2}} f(x) = \calF^{-1} \paren*{|\xi|^\alpha \hat{f}(\xi)}
\end{align}
with the Fourier transform $\hat{f}(\xi) = \calF [f] (\xi)$.
The fractional Laplacian is  a nonlocal operator in general
except for the standard Laplacian $\alpha=2$.

The fractional Schr\"odinger equation has found several applications in physics~\cite{gx06,kl13,lo15,zl15}.
Besides, from mathematical viewpoints, several properties such as the well-posedness
have been established~\cite{ch15,gh08,gh11}.
Further, there has been a growing interest
in constructing reliable and efficient numerical schemes.
For the conventional NLS equation, it has been proved and is now well-known that
structure-preserving schemes such as symplectic integrators, multi-symplectic integrators and invariants-preserving integrators, have substantial advantages
over general-purpose methods such as explicit Runge--Kutta-type methods with
 standard finite difference spatial discretization:
structure-preserving schemes often produce qualitatively better numerical solutions
over a long-time interval with a relatively large time step size~(see, e.g.~\cite{dfp81,fa12}).
Therefore, it is no wonder that
recent papers on the FNLS equation have mainly focused on the construction of structure-preserving schemes.
For example,
based on a Hamiltonian or generalized multi-symplectic structure,
symplectic or multi-symplectic schemes were presented in~\cite{wh18}.
The fractional Schr\"odinger equation has mass (probability) and energy conservation laws, and
several invariant(s)-preserving schemes have been developed: schemes preserving the mass~\cite{lg18,wh15b,wh16}
and schemes preserving both the mass and energy~\cite{dz16,lh19,lh17,wh15,wh16c}.
Most of these schemes are generalizations of existing schemes developed for the case $\alpha = 2$.
For example, the scheme \cite{lh19} can be seen as a generalization of the so-called relaxation scheme~\cite{be04} (see~\cite{bd18} for the theoretical analysis).

In this paper, we are concerned with invariants-preserving numerical schemes.
Most of the aforementioned numerical schemes preserving both invariants are nonlinear,
and thus computationally expensive.
Therefore, it is hoped that those schemes are linearized
without deteriorating the mass and energy conservations.
However, there are two challenges for this aim.
First, standard approaches to the linearization of nonlinear schemes often lose conservation properties.
In fact, linearly implicit schemes derived in~\cite{lg18,wh15b} only preserve the mass.
Second, even if we are succeeded in designing intended linearly implicit schemes,
they result in solving dense linear systems due to the fractional Laplacian, which is nonlocal.
Although some of the papers mentioned above contribute to the treatment of the
fractional Laplacian in the discrete settings, no matter how we discretize the
fractional Laplacian a matrix representing a discrete fractional Laplacian
is a dense matrix.
Thus, the computation of such dense linear systems is an important issue.
In fact,
the computational complexity of direct solvers applied to dense linear systems is of $\rmO(N^3)$,
where $N$ is the size of the system.
Even if we consider iterative solvers, each iteration usually requires
$\rmO(N^2)$ operations.

With these backgrounds,
we aim to derive a linearly implicit numerical scheme preserving discrete mass and energy simultaneously,
and then discuss the implementation issues for the proposed scheme.
In our problem setting,
imposing the periodic boundary condition, we mainly focus on the FNLS equation of the form~\cite{gh08}
\begin{align} \label{eq:fnls}
u_t =-\rmi \fl{\frac{\alpha}{2}} u + \rmi |u|^2 u, \quad x\in\bbT, \ t>0
\end{align}
with the initial condition $u(0,x) = u_0(x)$,
where
$\bbT = \bbR/ L \bbZ$ denotes the one-dimensional torus of length $L$.
The fractional Laplacian $\fl{\frac{\alpha}{2}}$ is now defined by
\begin{align} \label{eq:fl}
\fl{\frac{\alpha}{2}} u = \sum_{k\in\bbZ} |\mu k|^{\alpha}\hat{u}_k \rme^{\rmi\mu kx},
\end{align}
where $\hat{u}_k$ denotes the Fourier coefficients
\begin{align}
u = \sum_{k\in\bbZ} \hat{u}_k \rme^{\rmi\mu kx}, \quad
\hat{u}_k = \frac{1}{L} \int_{\bbT} u (x) \rme^{-\rmi \mu kx}\,\rmd x
\end{align}
and $\mu = 2\pi / L$.
The mass (probability) and energy for  the FNLS equation~\eqref{eq:fnls} are defined by
\begin{align}
\calM[u] &= \int_{\bbT} |u|^2 \,\rmd x, \label{eq:mass} \\
\calH[u] &= \int_{\bbT}
\paren*{-\left|\fl{\frac{\alpha}{4}}u \right|^2 + \frac{|u|^4}{2}}\,\rmd x, \label{eq:energy1}
\end{align}
respectively, and these quantities are constant along the solution~\cite{gh08}.

We employ the pseudo-spectral method
for the space discretization. In this approach, the treatment of the fractional Laplacian is rather straightforward,
as will be explained in Section~\ref{sec22}.
Time discretization of the proposed scheme is the same as the one proposed for the
coupled FNLS equation~\cite{wx14}.
However, our derivation is a bit more systematic,
and to illustrate this we also show that invariants-preserving linearly implicit schemes
can be constructed for the FNLS equation with stronger nonlinearity.
We note that our approach is motivated by the discrete variational derivative method~\cite{do11,fm11}.

When we implement the proposed scheme, we need to choose a linear solver carefully.
Krylov subspace methods seem suitable because
the multiplication of a vector by the coefficient matrix 
can be efficiently computed with $\rmO(N\log N)$ operations
thanks to the fast Fourier transform.
The coefficient matrix of the linear system appearing in the scheme is found to be complex and symmetric (thus, non-Hermitian), and varies as time passes.
For non-Hermitian systems, the Bi-CGSTAB method~\cite{va92} is a standard choice, but this method requires two matrix-vector
products per iteration.
As alternative choices, we test the conjugate orthogonal conjugate gradient (COCG) method and conjugate orthogonal conjugate residual (COCR) method,
which were specially designed for solving  complex symmetric linear systems~\cite{sz07,vm90} and require only a single matrix-vector
product per iteration.
We will observe that these methods work well
when a relatively small time step size is used. 
But since one of the advantages of structure-preserving schemes 
is that they can give qualitatively correct numerical solutions
with relatively large time step sizes, 
the convergence of these methods for the case
the time step size is not small enough is of interest.
Unfortunately, however, our preliminary experiments  show that
if the time step size is relatively large, the coefficient matrix tends to be ill-conditioned, and the number of matrix-vector
products required to reach convergence tends to become large even if the COCG or COCR method is employed
(although the results remain better than those by the Bi-CGSTAB method).
To address such situations, we consider preconditioning issues. 
We propose a certain preconditioner, which can be incorporated into the Bi-CGSTAB method.
Numerical experiments show that the number of iterations to reach convergence for the preconditioned Bi-CGSTAB method is almost independent of the size of the linear system.

The paper is organized as follows.
In Section~\ref{sec2}, our notation and several preliminary results are summarized.
In Section~\ref{sec3}, our intended scheme is presented, and its properties are
discussed. The preconditioning issues are also addressed.
Several numerical studies are conducted in Section~\ref{sec4},
and finally
concluding remarks are given in Section~\ref{sec5}.

\section{Preliminaries}
\label{sec2}

In this section, our notation and several preliminary results are summarized.

\subsection{Mass and energy preservations for the FNLS equation}

The mass and energy conservation laws follow in a straightforward way,
but we here sketch the proof since it will be mimicked in the discrete settings in Section~\ref{sec3}.

\begin{theorem}[e.g.~\cite{gh08}]
For the solution to the FNLS equation \eqref{eq:fnls}, it follows that
\begin{align}
\frac{\rmd}{\rmd t} \calM[u] = 0, \quad
\frac{\rmd}{\rmd t} \calH[u] = 0.
\end{align}
\end{theorem}

\begin{proof}
First, we prove the mass preservation.
\begin{align*}
\frac{\rmd}{\rmd t}\calM[u]
&= \int_{\bbT} \paren*{u_t \olu + u\olu_t}\,\rmd x
= \int_{\bbT} \left[ \paren*{-\rmi \fl{\frac{\alpha}{2}}u + \rmi |u|^2 u}\olu
+ u \paren*{\rmi \fl{\frac{\alpha}{2}}\olu - \rmi |u|^2 \olu}\right]\,\rmd x \\
&= \rmi \int_{\bbT} \left[ -\paren*{ \fl{\frac{\alpha}{2}} u}\olu
+  u \fl{\frac{\alpha}{2}} \olu\right] \,\rmd x \\
&= \rmi \int_{\bbT} \left[-\paren*{\fl{\frac{\alpha}{4}}u}\paren*{\fl{\frac{\alpha}{4}}\olu}
+ \paren*{\fl{\frac{\alpha}{4}}u}\paren*{\fl{\frac{\alpha}{4}}\olu}\right] \, \rmd x
= 0.
\end{align*}
The first equality is just the chain rule.
The fourth equality is due to the integration-by-parts formula:
for any $L$-periodic complex-valued functions $u$ and $v$, we have
\begin{align}
\int_{\bbT} \paren*{\fl{\frac{\alpha}{2}} u } \overline{v} \,\rmd x
&= \int_{\bbT} \paren*{\fl{\frac{\alpha}{4}} u }  \paren*{\fl{\frac{\alpha}{4}} \overline{v} }  \,\rmd x . \label{ibp1}
\end{align}
Next, we prove the energy preservation.
\begin{align}
\frac{\rmd}{\rmd t} \calH[u]
&=
\int_{\bbT}
\left[
-\paren*{\fl{\frac{\alpha}{4}}u_t}  \paren*{\fl{\frac{\alpha}{4}}\olu}
-\paren*{\fl{\frac{\alpha}{4}}u}  \paren*{\fl{\frac{\alpha}{4}}\olu_t}
+ u_t |u|^2 \olu + |u|^2 u \olu_t
\right] \,\rmd x \\
&=
\int_{\bbT}
\left[
-u_t \paren*{\fl{\frac{\alpha}{2}} \olu - |u|^2 \olu}
-\paren*{\fl{\frac{\alpha}{2}} u - |u|^2 u} \olu_t
\right]\,\rmd x \\
&=
\int_{\bbT}
\left[
-u_t (-\rmi \olu_t) - (\rmi u_t) \olu_t
\right]\,\rmd x
=
0.
\end{align}
\end{proof}

\subsection{Discrete settings}
\label{sec22}

\subsubsection{Notation}

The period $L$ is divided by $N$ equal grids, which means $\Delta x = L/N$.
We denote the numerical solution for $U(n\Delta t, k\Delta x)$
by $U_k^{(n)}$.
We often write the solutions as a vector $\bU^{(n)} = (U_0^{(n)}, U_1^{(n)},\dots, U_{N-1}^{(n)})^\top$.
To treat the periodic boundary condition,
we consider $\{ U_k \}_{k\in\bbZ}$, an infinitely long vector,
and then its $N$-dimensional restriction by the discrete periodic boundary condition:
$U_k^{(n)} = U_{k \ \text{mod} \ N}^{(n)}$ for all $k \in \bbZ$.
The space to which such periodic vectors belong is denoted by
$X_\rmd = \{ \bU = \{ U_k \}_{k\in\bbZ} \ | \ U_k \in \bbC, \ U_k = U_{k \ \text{mod} \ N} \text{, for all }k\in\bbZ \}$.

\subsubsection{Discrete fractional Laplacian}

We define a discrete fractional Laplacian and show its properties.
For simplicity, we assume that $N$ is an odd number
keeping in mind that the following discussion can be extended to an even number $N$
straightforwardly.
We employ the Fourier pseudo-spectral approach for the space discretization.

We define a function space $S_N$ by
\begin{align}
S_N = \mathspan \{g_j (x), \ j = 0,1,\dots,N-1  \},
\end{align}
where $g_j (x)$ is a trigonometric polynomial defined by
\begin{align}
g_j (x) = \frac{1}{N} \sum_{p = -\frac{N-1}{2}}^{\frac{N-1}{2}} \rme ^{\rmi \mu p (x-x_j)},
\quad \mu = 2\pi/L.
\end{align}
Note that $g_j(x_l) = \delta _{j,l}$, where  $\delta _{j,l}$ denotes the Kronecker delta.
We define the interpolation operator $I_N: L^2 (\bbT) \to S_N$ by
\begin{align} \label{eq:itlf}
I_N u(x) = \sum_{j=0}^{N-1} u_j g_j (x)
\end{align}
so that $I_N u(x_j) = u_j$,
where $u_j = u(x_j)$ and $L^2 (\bbT)$ denotes the set of square integrable functions on $\bbT$.
Applying the fractional Laplacian $\fl{\frac{\alpha}{2}}$
to the interpolated function \eqref{eq:itlf} yields
\begin{align}
\fl{\frac{\alpha}{2}} I_N u(x)
=
\frac{1}{N}
\sum_{j=0}^{N-1} u_j  \sum_{p = -\frac{N-1}{2}}^{\frac{N-1}{2}} |\mu p|^\alpha \rme ^{\rmi \mu p (x-x_j)},
\end{align}
and thus
\begin{align*}
\fl{\frac{\alpha}{2}} I_N u(x_k)
&
=\frac{1}{\sqrt{N}} \sum_{p=0}^{N-1} d_p
\paren*{\frac{1}{\sqrt{N}} \sum_{j=0}^{N-1} u_j \rme ^{-\frac{2\pi \rmi jp}{N}}}
\rme^{\frac{2\pi\rmi pk}{N}},
\end{align*}
where
\begin{align*}
d_p = \begin{cases}
|\mu p|^\alpha, & 0\leq p \leq \frac{N-1}{2},\\
|\mu (p-N)|^\alpha, & \frac{N+1}{2} \leq p \leq N-1.
\end{cases}
\end{align*}
We then define a discrete fractional Laplacian $\dfl{\frac{\alpha}{2}}$ by
\begin{align*}
\paren*{\dfl{\frac{\alpha}{2}} \bU} _ k =
\frac{1}{\sqrt{N}} \sum_{p=0}^{N-1} d_p
\paren*{\frac{1}{\sqrt{N}} \sum_{j=0}^{N-1} U_j \rme ^{-\frac{2\pi \rmi jp}{N}}}
\rme^{\frac{2\pi\rmi pk}{N}}.
\end{align*}
By using the notation of the discrete Fourier transform and its inverse:
\begin{align}
(\calF_\rmd \bU )_k = \frac{1}{\sqrt{N}} \sum_{j=0}^{N-1}U_j \rme ^{-\frac{2\pi \rmi jk}{N}},
\quad
(\calF_\rmd^{-1} \hat{\bU} )_j = \frac{1}{\sqrt{N}} \sum_{k=0}^{N-1} \hat{U}_k
\rme ^{\frac{2\pi \rmi jk}{N}},
\end{align}
the discrete fractional Laplacian can also be expressed as
\begin{align}
\dfl{\frac{\alpha}{2}} \bU = \calF_\rmd^{-1} D_\alpha \calF_\rmd \bU,
\end{align}
where
\begin{align}\label{def:da}
D_\alpha =
\begin{bmatrix}
0 &  &  &  &  &  &  \\
   & |\mu \cdot 1|^\alpha &  & &  &  &  \\
   &  & \ddots &  &  &  &  \\
   &  &  & \left|\mu  \paren*{\frac{N-1}{2}}\right|^\alpha &  &  &  \\
   &  &  &  & \left|\mu \paren*{\frac{N-1}{2}}\right|^\alpha & &  \\
   &  &  &  &  & \ddots &  \\
   &  &  &  &  &  & |\mu \cdot 1|^\alpha
\end{bmatrix}.
\end{align}
We shall abuse notation and write $\dfl{\frac{\alpha}{2}} U_k $ instead of $\paren*{\dfl{\frac{\alpha}{2}} \bU }_k$.

The next lemma shows that the discrete fractional Laplacian $\dfl{\frac{\alpha}{2}}$
is a Hermitian operator.
This property will be used to prove the unique solvability of our proposed scheme.

\begin{lemma}
\label{lem:her}
The discrete fractional Laplacian $\dfl{\frac{\alpha}{2}}$ is Hermitian, i.e.
self-adjoint:
\begin{align}
\paren*{\dfl{\frac{\alpha}{2}}}^\dagger = \dfl{\frac{\alpha}{2}},
\end{align}
where the symbol $\dagger$ denotes the Hermitian adjoint.
\end{lemma}

\begin{proof}
The operator $D_\alpha$ defined by \eqref{def:da} is
Hermitian ($D_\alpha^\dagger = D_\alpha$)
since it is a real diagonal matrix.
Due to the unitarity of the discrete Fourier transform $\calF_\rmd$: $\calF_\rmd^{-1} = \calF_\rmd^\dagger $,
we have
$\paren*{\calF_\rmd^{-1}}^\dagger = \paren*{\calF_\rmd^\dagger}^\dagger = \calF_\rmd$.
Therefore, it follows that
\begin{align}
\paren*{\dfl{\frac{\alpha}{2}}}^\dagger
= \paren*{\calF_\rmd^{-1}D_\alpha \calF_\rmd}^\dagger
= \calF_\rmd^\dagger D_\alpha^\dagger \paren*{\calF_\rmd^{-1}}^\dagger
=\calF_\rmd^{-1} D_\alpha \calF_\rmd
= \dfl{\frac{\alpha}{2}}.
\end{align}
\end{proof}

Note that $\dfl{\frac{\alpha}{4}}$ is also Hermitian, and $\dfl{\frac{\alpha}{2}} = \dfl{\frac{\alpha}{4}} \circ \dfl{\frac{\alpha}{4}}$.
Thus, as an immediate consequence of Lemma~\ref{lem:her}, we have the
following summation-by-parts formula
that corresponds to the integration-by-parts formula \eqref{ibp1}.

\begin{lemma}
For any two vectors $\bU, \bV \in X_\rmd$, we have
\begin{align} \label{sbpf}
\Delta x \sum_{k=0}^{N-1}
\paren*{\dfl{\frac{\alpha}{2}} U_k } \overline{V}_k
=
\Delta x \sum_{k=0}^{N-1}
\paren*{\dfl{\frac{\alpha}{4}} U_k }
\paren*{\dfl{\frac{\alpha}{4}} \overline{V}_k }.
\end{align}
\end{lemma}

\begin{remark}
\label{rm:real}

It can be readily checked that $\calF_\rmd ^{-1} D_\alpha \calF_\rmd$
is a real matrix,
but this property does not hold if we replace $D_\alpha$ with
a general real diagonal matrix.
\end{remark}

\section{Linearly implicit scheme}
\label{sec3}

In this section, we present a linearly implicit scheme for the FNLS equation
\eqref{eq:fnls}, show several properties of the scheme, and discuss the implementation with an emphasis on  preconditioning issues.

\subsection{Linearly implicit scheme}
We propose the following linearly implicit scheme:
given an initial approximation $\bU^{(0)}\in X_\rmd$ and starting approximation
$\bU^{(1)} \in X_\rmd$, we compute $\bU^{(n)}$ $(n=2,3,\dots )$ by
\begin{align} \label{scheme1}
\frac{U_k^{(n+1)}-U_k^{(n-1)}}{2 \Delta t}
= -\rmi \dfl{\frac{\alpha}{2}} \paren*{\frac{U_k^{(n+1)} + U_k^{(n-1)}}{2}}
+ \rmi \left| U_k^{(n)}\right| ^2 \paren*{\frac{U_k^{(n+1)}+U_k^{(n-1)}}{2}} & \\
\quad (k = 0 ,\dots, N-1).&
\end{align}
The initial approximation is usually set to
$U_k^{(0)} = u_0(k\Delta x)$.
The starting value $\bU^{(1)}$ can be prepared in several ways, and this issue will be discussed later.
The scheme \eqref{scheme1} can be written in the following form:
\begin{align} \label{scheme2}
\paren*{I + \rmi \Delta t \paren*{\calF_\rmd^{-1} D_\alpha \calF_\rmd - D(\bU^{(n)})}}
\bU^{(n+1)} =
\paren*{I - \rmi \Delta t \paren*{\calF_\rmd^{-1} D_\alpha \calF_\rmd - D(\bU^{(n)})}}
\bU^{(n-1)},
\end{align}
where $D(\bU) = \diag (|U_0|^2, |U_1|^2 , \dots , |U_{N-1}|^2)$;
therefore a linear system with the coefficient matrix $\displaystyle I + \rmi \Delta t \paren*{\calF_\rmd^{-1} D_\alpha \calF_\rmd - D(\bU^{(n)})}$ needs to be solved per each time step.
We note that the coefficient matrix is complex and symmetric
because $\calF_\rmd ^{-1} D_\alpha \calF_\rmd$ is a real
symmetric matrix according to Lemma~\ref{lem:her} and Remark~\ref{rm:real}.

Below, we show that the scheme~\eqref{scheme1} preserves the following discrete mass and energy:
\begin{align}
\Md ( \bU) &= \Delta x \sum_{k} \left|U_k\right|^2, \label{dn} \\
\Hd (\bU,\bV) &= \Delta x  \sum_{k} \paren*{
-
\frac{\left|  \paren*{\fl{\frac{\alpha}{4}} \bU}_k \right| ^2
+ \left|  \paren*{\fl{\frac{\alpha}{4}} \bV}_k \right| ^2}{2}
 + \frac{\left| U_k \right|^2 \left| V_k \right|^2}{2}}. \label{de}
\end{align}
Note that though the discrete energy depends on two solution vectors,
it seems a reasonable one since if $\bV = \bU$ it becomes a more straightforward definition
\begin{align}
\tilde{\calH}_\rmd (\bU):= \calH_\rmd(\bU,\bU) =
\Delta x  \sum_{k} \paren*{
-
\left|  \paren*{\fl{\frac{\alpha}{4}} \bU}_k \right| ^2
 + \frac{\left| U_k \right|^4}{2}}. \label{de1}
\end{align}

\begin{theorem} \label{th:scheme1}
The scheme \eqref{scheme1}
conserves the discrete mass $\Md$ and the discrete energy $\Hd$ in the sense that
\begin{align}
\Md \paren*{\bU^{(n+1)}} = \Md \paren*{\bU^{(n-1)}}, \quad
\Hd \paren*{\bU^{(n+1)}, \bU^{(n)} }
= \Hd \paren*{\bU^{(n)}, \bU^{(n-1)} } , \quad
n=1,2,\dots .
\end{align}
\end{theorem}

\begin{proof}
First, we prove the discrete mass preservation.
\begin{align}
&\frac{1}{2\Delta t} \paren*{\Md \paren*{\bU^{(n+1)}} - \Md \paren*{\bU^{(n-1)}} } \\
&\quad=
	\Delta x \sum_{k=0}^{N-1} \left[
	\paren*{\frac{U_k^{(n+1)}-U_k^{(n-1)}}{2\Delta t}}
	\paren*{\overline{\frac{U_k^{(n+1)} + U_k^{(n-1)}}{2}}}  
	+
	\paren*{\frac{U_k^{(n+1)} + U_k^{(n-1)}}{2}}
	\paren*{\overline{\frac{U_k^{(n+1)}-U_k^{(n-1)}}{2\Delta t}}}
	\right] \\
&\quad =
	\Delta x \sum_{k=0}^{N-1} \left[
	\paren*{
	-\rmi \dfl{\frac{\alpha}{2}} \paren*{\frac{U_k^{(n+1)} + U_k^{(n-1)}}{2} }
	+ \rmi \left| U_k^{(n)}\right|^2 \paren*{\frac{U_k^{(n+1)}-U_k^{(n-1)}}{2\Delta t}}}
	\paren*{\overline{\frac{U_k^{(n+1)} + U_k^{(n-1)}}{2}}}
	\right. \\
& \quad \phantom{	=(\Delta x)^d \sum_{k} }
	+
	\left.
	\paren*{\frac{U_k^{(n+1)} + U_k^{(n-1)}}{2}}
	\paren*{
	\rmi \dfl{\frac{\alpha}{2}} \paren*{\overline{\frac{U_k^{(n+1)} + U_k^{(n-1)}}{2} }}
	- \rmi \left| U_k^{(n)}\right|^2 \paren*{\overline{\frac{U_k^{(n+1)}-U_k^{(n-1)}}{2\Delta t}}}
	}
	\right] \\
& \quad =
	\rmi \Delta x \sum_{k=0}^{N-1} \left[
	-
	\paren*{\dfl{\frac{\alpha}{2}} V_k^{(n)}} \overline{V_k^{(n)}}
	+
	V_k \paren*{\dfl{\frac{\alpha}{2}} \overline{V_k^{(n)}}}
	\right] \\
& \quad =
	\rmi \Delta x \sum_{k=0}^{N-1} \left[
	-
	\paren*{\dfl{\frac{\alpha}{4}} V_k^{(n)}}
	\paren*{\dfl{\frac{\alpha}{4}} \overline{V_k^{(n)}}}
	+
	\paren*{\dfl{\frac{\alpha}{4}} \overline{V_k^{(n)}}}
	\paren*{\dfl{\frac{\alpha}{4}} V_k^{(n)}}
	\right]
	=0,
\end{align}
where $V_k^{(n)} = \paren*{U_k^{(n+1)} + U_k^{(n-1)}}/2$.
The first equality is just the factorization,
which corresponds to the chain rule.
In the second equality, we have substituted the scheme~\eqref{scheme1}.
The fourth equality is due to the summation-by-parts formula \eqref{sbpf}.

Next, we prove the discrete energy preservation.
\begin{align}
&\frac{1}{\Delta t} \paren*{\Hd \paren*{\bU^{(n+1)}, \bU^{(n)} } - \Hd \paren*{\bU^{(n)}, \bU^{(n-1)} }} \\
&\quad=
	\Delta x\sum_{k=0}^{N-1}
	\left[
	- \frac{\left| \dfl{\frac{\alpha}{4}} U_k^{(n+1)} \right|^2
	- \left| \dfl{\frac{\alpha}{4}} U_k^{(n-1)} \right|^2}{2\Delta t} 
	+
	\frac{\left| U_k^{(n)} \right|^2
	\paren*{\left| U_k^{(n+1)} \right|^2 - \left| U_k^{(n-1)} \right|^2}
	}{2\Delta t}
	\right] \\
&\quad=
	\Delta x\sum_{k=0}^{N-1}
	\left[
	-
	\paren*{\dfl{\frac{\alpha}{4}} \paren*{\frac{U_k^{(n+1)}-U_k^{(n-1)}}{2\Delta t}}}
	\paren*{\dfl{\frac{\alpha}{4}} \paren*{\overline{\frac{U_k^{(n+1)}+U_k^{(n-1)}}{2}}}} \right. \\
&\quad	\phantom{	=(\Delta x)^d \sum_{k} }
	-
	\paren*{\dfl{\frac{\alpha}{4}} \paren*{\frac{U_k^{(n+1)}+U_k^{(n-1)}}{2}}}
	\paren*{\dfl{\frac{\alpha}{4}} \paren*{\overline{\frac{U_k^{(n+1)}-U_k^{(n-1)}}{2\Delta t}}}}
	\\
& \quad \phantom{	=(\Delta x)^d \sum_{k} }
	+
	\left| U_k^{(n)} \right| ^2
	\paren*{\frac{U_k^{(n+1)}-U_k^{(n-1)}}{2\Delta t}}
	\paren*{\overline{\frac{U_k^{(n+1)}+U_k^{(n-1)}}{2}}}\\
& \quad \phantom{	=(\Delta x)^d \sum_{k} } \left.
	+
	\left| U_k^{(n)} \right| ^2
	\paren*{\frac{U_k^{(n+1)}+U_k^{(n-1)}}{2}}
	\paren*{\overline{\frac{U_k^{(n+1)}-U_k^{(n-1)}}{2\Delta t}}}
	\right] \\
&\quad=
	\Delta x\sum_{k=0}^{N-1}
	\left[
	-
	\paren*{\frac{U_k^{(n+1)}-U_k^{(n-1)}}{2\Delta t}}
	\paren*{
	\dfl{\frac{\alpha}{2}} \overline{V_k^{(n)}} - \left| U_k^{(n)} \right| ^2 \overline{V_k^{(n)}}
	} \right. 
	\\
& \quad \phantom{	=(\Delta x)^d \sum_{k} } \left.
	-
	\paren*{
	\dfl{\frac{\alpha}{2}} V_k^{(n)} - \left| U_k^{(n)} \right| ^2 V_k^{(n)}
	}
	\paren*{\overline{\frac{U_k^{(n+1)}-U_k^{(n-1)}}{2\Delta t}}}
	\right] \\
&\quad=
	\Delta x \sum_{k=0}^{N-1}
	\left[
	-
	\paren*{\frac{U_k^{(n+1)}-U_k^{(n-1)}}{2\Delta t}}
	\paren*{-\rmi
	\paren*{\overline{\frac{U_k^{(n+1)}-U_k^{(n-1)}}{2\Delta t}}}} \right. 
	\\
& \quad \phantom{	=(\Delta x)^d \sum_{k} } \left. 
	-
	\paren*{\rmi
	\paren*{\frac{U_k^{(n+1)}-U_k^{(n-1)}}{2\Delta t}}
	}
	\paren*{\overline{\frac{U_k^{(n+1)}-U_k^{(n-1)}}{2\Delta t}}}
	\right] \\
&=0.
\end{align}
In the second equality, we have used the identity
$aA - bB = (a+b)(A-B)/2 + (a-b)(A+B)/2$,
and in the third equality, we have substituted the scheme~\eqref{scheme1}.
\end{proof}

\begin{remark}
In the above discussion, defining the discrete quantities is the most crucial part, since if we employ other definitions, the corresponding schemes vary and might be nonlinear. In particular, the key to defining \eqref{de} is that it is at most quadratic with respect to both vectors $\bU$ and $\bV$. For more details, we refer the reader to~\cite{do11,fm11,mf01}.
\end{remark}

In general, while the computational complexity of the time integration of
linearly implicit schemes is much cheaper than that of nonlinear schemes,
the former often suffers from the strict restriction on step sizes to ensure the solvability (see, e.g.~\cite{mm12}).
Fortunately, the following theorem shows that the scheme \eqref{scheme1} is unconditionally solvable without any restriction on the time and space step sizes.

\begin{theorem}[Unique solvability]
For $n\geq 1$,
let
$\bU^{(n-1)}\in X_\rmd$  and
$\bU^{(n)}\in X_\rmd$ be given.
Then,
the scheme \eqref{scheme1} has a unique solution $\bU^{(n+1)}$
without any restriction on the time and space step sizes.
\end{theorem}

\begin{proof}
We show that the operator $I + \rmi \Delta t \paren*{\calF_\rmd^{-1} D_\alpha \calF_\rmd - D(\bU^{(n)})}$
in \eqref{scheme2} is nonsingular,
independently of $\bU^{(n)}$.
Since $\calF_\rmd^{-1} D_\alpha \calF_\rmd$ is Hermitian due to Lemma~\ref{lem:her},
its eigenvalues are all real.
Note that $D (\bU^{(n)})$ is a diagonal matrix and its elements are all real.
Therefore, the real part of all eigenvalues of $I + \rmi \Delta \paren*{\calF_\rmd^{-1} D_\alpha \calF_\rmd - D(\bU^{(n)})}$
is $1$, which indicates that the operator is nonsingular.
\end{proof}

\subsection{Starting procedure}
The starting approximation $\bU^{(1)}$ can be computed in many different ways.
Among several possibilities,
 in this paper we employ the Crank--Nicolson scheme
\begin{align} \label{scheme3}
\frac{U_k^{(1)}-U_k^{(0)}}{\Delta t}
= -\rmi \dfl{\frac{\alpha}{2}} \paren*{\frac{U_k^{(1)} + U_k^{(0)}}{2}}
+ \rmi \paren*{\frac{\left| U_k^{(1)}\right| ^2 + \left| U_k^{(0)}\right| ^2}{2}} \paren*{\frac{U_k^{(1)}+U_k^{(0)}}{2}} &\\
\quad (k = 0 ,\dots, N-1). &
\end{align}
As is the case with the finite difference scheme applied to the case $\alpha = 2$ \cite{dfp81},
it is easily verified that
the scheme~\eqref{scheme3} conserves the discrete mass $\calM_\rmd$ and the discrete energy $\tilde{\calH}_\rmd$
in the sense that
\begin{align}
\calM_\rmd \paren*{\bU^{(1)}} = \calM_\rmd \paren*{\bU^{(0)}}, \quad
\tilde{\calH}_\rmd \paren*{\bU^{(1)} }
= \tilde{\calH}_\rmd \paren*{\bU^{(0)} } .
\end{align}

As long as this nonlinear scheme~\eqref{scheme3}
or a mass-preserving one-step scheme is used to obtain the starting approximation
$\bU^{(1)}$,
it follows that $\Md (\bU^{(n)}) = \Md (\bU^{(0)})$ ($n=0,1,2,\dots$)
for the solution to the linearly implicit scheme~\eqref{scheme1}.
However, even if the nonlinear scheme~\eqref{scheme3} is employed to obtain $\bU^{(1)}$,
$\tilde{\calH}_\rmd (\bU^{(n)}) \neq \tilde{\calH}_\rmd (\bU^{(0)})$ in general.
In Section~\ref{subsec:energy-pre}, we shall numerically investigate to what extent $\tilde{\calH}_\rmd (\bU^{(n)})$ remains close to $\tilde{\calH}_\rmd(\bU^{(0)})$.

\begin{remark}
Nonlinear schemes preserving either the mass or energy
can also be derived, as presented in~\cite{ce12} for the case $\alpha = 2$.
There, the idea is to apply the average vector field method~\cite{qm08} to a semi-discrete scheme which is obtained based on a variational structure.
\end{remark}

\subsection{Linearly implicit schemes for the FNLS with strong nonlinearity}
\label{sec3.3}

In addition to the FNLS equation with cubic nonlinearity,
which is the main focus of this paper,
we briefly explain that
similar structure-preserving linearly implicit schemes
can be constructed for the FNLS equation with stronger nonlinearity.
We consider
\begin{align} \label{eq:fnlsn}
u_t =-\rmi \fl{\frac{\alpha}{2}} u + \rmi |u|^{2\rho} u, \quad x\in\bbT, \ t>0,
\end{align}
where $\rho \in \{1,2,\dots\}$.
For this equation, while the mass \eqref{eq:mass} remains an invariant,
the definition of the energy is modified to
\begin{align}
\calH^\rho (u) = \int_{\bbT}
\paren*{-\left|\fl{\frac{\alpha}{4}}u \right|^2 + \frac{|u|^{2\rho+2}}{\rho + 1}}\,\rmd x. \label{eq:energy2}
\end{align}
Note that this reduces to \eqref{eq:energy1} when $\rho=1$.
We have already discussed the case $\rho = 1$ in the above subsections.
Observe that the scheme \eqref{scheme1} is a two-step method.
If $\rho \geq 2$, an intended scheme requires more steps so that the resulting scheme is linear
in terms of the unknown solution vector.
Our intended scheme is defined by
\begin{align} \label{scheme4}
&\frac{U_k^{(n+1)}-U_k^{(n-\rho)}}{(\rho+1) \Delta t} \\
& \quad  = -\rmi \dfl{\frac{\alpha}{2}} \paren*{\frac{U_k^{(n+1)} + U_k^{(n-\rho)}}{2}}
+ \rmi \left| U_k^{(n)}\right| ^2 \cdots \left| U_k^{(n-\rho+1)}\right| ^2 \paren*{\frac{U_k^{(n+1)}+U_k^{(n-\rho)}}{2}}  \\
& \hspace*{10.5cm} (k = 0 ,\dots, N-1),
\end{align}
which is a ($\rho+1$)-step method.
This scheme is mass-preserving $\Md (\bU^{(n+1)}) = \Md (\bU^{(n-\rho)})$
and further energy-preserving in the sense that
$\Hd^\rho \paren*{\bU^{(n+1)},\dots, \bU^{(n-\rho+1)} }
= \Hd^\rho \paren*{\bU^{(n)},\dots, \bU^{(n-\rho)} }$,
where the discrete energy $\Hd^\rho$ is defined by
\begin{align}
&\Hd^\rho \paren*{\bU^{(n)},\dots, \bU^{(n-\rho)} } \\
& \quad =
\Delta x \sum_{k=0}^{N-1}
\paren*{
-
\frac{\left|  \paren*{\fl{\frac{\alpha}{4}} \bU^{(n)}}_k \right| ^2
+\cdots + \left|  \paren*{\fl{\frac{\alpha}{4}} \bU^{(n-\rho)}}_k \right| ^2}{\rho+1}
 + \frac{\left| U_k^{(n)} \right|^2 \cdots   \left| U_k^{(n-\rho)} \right|^2}{\rho+1}}.
\end{align}
Note that this quantity is quadratic with respect to each vector $\bU^{(i)}$ ($i=n-\rho,\dots,n$).
The proof of the preservations is similar to that of Theorem~\ref{th:scheme1} and thus omitted.

\subsection{Preconditioning}
\label{subsec:pcond}
The coefficient matrix of the linear system~\eqref{scheme2} is complex and symmetric because $\calF_\rmd ^{-1} D_\alpha \calF_\rmd$ is a real
symmetric matrix according to Lemma~\ref{lem:her} and Remark~\ref{rm:real}.
It is hoped that the linear system is solved efficiently per each time step.

As explained in Section~\ref{sec1},
the multiplication of a vector by the coefficient matrix can be
efficiently computed with $\rmO(N\log N)$ operations
thanks to the fast Fourier transform,
and hence Krylov subspace methods seem suitable choices for
solving \eqref{scheme2}.
For non-Hermitian systems, the Bi-CGSTAB method~\cite{vm90} is a standard choice, but this requires two matrix-vector products per iteration. 
As alternative choices, we test the conjugate orthogonal conjugate gradient (COCG) method~\cite{va92} and the conjugate orthogonal conjugate  residual
(COCR) method~\cite{sz07}, which were specially designed for solving complex and symmetric linear systems.
These methods require only a single matrix-vector product
per iteration.
As will be seen later, they work better than the Bi-CGSTAB method.
However, when relatively large step size or large $N$ is used, 
both the COCG/COCR and Bi-CGSTAB methods tend to require a large number of iterations because the coefficient matrix becomes ill-conditioned,
and this behaviour is problematic especially when we consider multi-dimensional problems.
Therefore, it is worth considering the preconditioning issues for one-dimensional problems.

We need to understand why the convergence behaviour is deteriorated (some numerical results will be illustrated in Section~\ref{subsec:peformpre}).
To simplify the notation, we rewrite \eqref{scheme2} as
\begin{align} \label{eq:ls1}
\paren*{I + \rmi \Delta t \paren*{\calF_\rmd^{-1} D_\alpha \calF_\rmd - D(\bU)}}
\bx =
\bb
\end{align}
where $D(\bU) = \diag (|U_0|^2, |U_1|^2 , \dots , |U_{N-1}|^2)$
and $D_\alpha$ is given in \eqref{def:da}.
The smallest and largest eigenvalues of the sum of the first two terms
$I + \rmi \Delta t\calF_\rmd^{-1} D_\alpha \calF_\rmd$
are $1$ and $1+\rmi \Delta t |\mu (N-1)/2|^\alpha$, respectively,
which indicates that the sum of the first two terms tends to be
ill-conditioned for large $\alpha$, $N$ and $\Delta t$.
On the other hand, since the third term, i.e. the diagonal matrix $\rmi \Delta t D(\bU)$, represents the shape of the solution, its condition
number remains nearly unaffected by the choice of $N$.
Therefore, the third term can be seen as a perturbation.
As will be illustrated in Section~\ref{subsec:peformpre},
the convergence is deteriorated for large $\alpha$, $N$ and $\Delta t$, and 
the sum of the first two terms is certainly the cause of the deterioration. Note that for large $\Delta t$ the influence of the last term could be substantial, and this influence further worsens the convergence.

Below, we consider the preconditioning issues.
Our idea is a combination of a certain variable transformation
and preconditioner.

For $A\bx = \bb$,
the preconditioned COCG and preconditioned Bi-CGSTAB methods with a matrix (preconditioner) $M$
are summarized in Algorithms~\ref{algo:pcocg} and~\ref{algo:pbicgstab},
where $(\bx,\by) = \overline{\bx}^\top \by$.
The preconditioner $M$ is ideally chosen such that
$M^{-1}A$ has a smaller condition number than $A$.
The simplest approach  is to extract the diagonal part of the coefficient matrix as a preconditioner $M$.
However, in our situation,
each diagonal element in the coefficient matrix of \eqref{eq:ls1} has almost the same value, because all diagonal elements of the dominant part
$I + \rmi \Delta t \calF_\rmd^{-1} D_\alpha \calF_\rmd$
are the same.
Therefore, this approach does not improve the convergence behaviour (in fact our preliminary numerical experiments
support this discussion).

\begin{algorithm}
\label{algo:pcocg}
\caption{Preconditioned COCG method}
$\bx_0$ is an initial guess, $\br_0 = \bb-A\bx_0$\\
set $\bp_{-1}=0$, $\beta_{-1}=0$\\
\For{$n=0,1,\dots$ until $\|\br_n\| \leq \varepsilon \| \bb \|$}{
$\bp_n = M^{-1}\br_n + \beta _{n-1}\bp_{n-1}$\\
$\displaystyle \alpha _n = \frac{(\overline{\br}_n, M^{-1}\br_n)}{(\overline{\bp}_n, A\bp_n)} $\\
$\bx_{n+1}=\bx_n+\alpha \bp_n$\\
$\br_{n+1} = \br_n -\alpha A\bp_n$\\
$\displaystyle \beta_n = \frac{(\overline{\br}_{n+1}, M^{-1}\br_{n+1})}{(\overline{\br}_n, M^{-1}\br_n)}$
}
\end{algorithm}

\begin{algorithm}
\label{algo:pbicgstab}
\caption{Preconditioned Bi-CGSTAB method}
$\bx_0$ is an initial guess, $\br_0 = \bb-A\bx_0$\\
set $\tilde{\br} = \br_0$\\
\For{$n=0,1,\dots$ until $\|\br_n\| \leq \varepsilon \| \bb \|$}{
$\rho_n =(\tilde{\br},\br_n) $\\
    \uIf{$n=0$}{
      $\bp_{n+1}=\br_n$
    }
    \Else{
      $\displaystyle \beta_n = \frac{\rho_n \alpha_n}{\rho_{n-1} \omega_n}$\\
      $\bp_{n+1} = \br_n + \beta_n (\bp_n - \omega_n\bv_n)$
    }
    $\hat{\bp}_{n+1} = M^{-1}\bp_{n+1}$\\
    $\bv_{n+1} = A \hat{\bp}_{n+1}$\\
    $\displaystyle \alpha_{n+1} = \frac{\rho_n}{(\tilde{\br},\bv_{n+1})}$\\
    $\bs_{n+1} = \br_n - \alpha_{n+1}\bv_{n+1}$\\
    $\hat{\bs}_{n+1} = M^{-1} \bs_{n+1}$\\
    $\bt_{n+1} = A \hat{\bs}_{n+1}$\\
    $\displaystyle \omega_{n+1} = \frac{(\bt_{n+1},\bs_{n+1})}{(\bt_{n+1},\bt_{n+1})}$\\
    $\bx_{n+1} = \bx_{n} +\alpha_{n+1}\hat{\bp}_{n+1} + \omega_{n+1}\hat{\bs}_{n+1}$\\
    $\br_{n+1} = \bs_{n+1} - \omega_{n+1}\bt_{n+1}$
}
\end{algorithm}

Let us consider a variable transformation $\by = \calF_\rmd \bx $.
By a similarity transformation, the linear system \eqref{eq:ls1} can be rewritten as
\begin{align}\label{eq:ls2}
\paren*{I + \rmi \Delta t D_\alpha  - \rmi \Delta t \calF_\rmd D(\bU)  \calF_\rmd^{-1} }
\by
 =
 \calF_\rmd \bb,
 \quad
 \by = \calF_\rmd \bx.
\end{align}
For the solution to the transformed linear system \eqref{eq:ls2},
the following property holds.
\begin{proposition}
  The relative 2-norm residual for \eqref{eq:ls2} coincides with that for \eqref{eq:ls1}.
\end{proposition}
\begin{proof}
  Let $\tilde{\bx}$ and $\tilde{\by}$ be approximations to \eqref{eq:ls1} and \eqref{eq:ls2}, respectively, with
  the relation $\tilde{\by} = \calF_\rmd \tilde{\bx}$.
  The exact solutions are denoted by $\bx_\ast$ and $\by_\ast$.
  Then, it follows that
  \begin{align}
  \frac{\| \tilde{\by} - \by_\ast \|}{\| \calF_\rmd \bb\|}
  = \frac{\| \calF_\rmd (\tilde{\bx} - \bx_\ast) \|}{\| \calF_\rmd \bb\|}
  = \frac{\| \tilde{\bx} - \bx_\ast \|}{\|\bb\|} 
  \end{align}
  due to the unitarity of $\calF_\rmd$.
\end{proof}

This property indicates that we only have to monitor the relative error for the transformed linear system \eqref{eq:ls2} during the iterations
and calculate $\tilde{\bx} = \calF_\rmd^{-1}\tilde{\by}$ only after the error meets the convergence criteria.
As a preconditioner, we use
\begin{align} \label{preconditioner}
M = I + \rmi \Delta t D_\alpha.
\end{align}

It should be noted that the coefficient matrix of the
transformed linear system \eqref{eq:ls2}
is still complex but no longer symmetric
because
$\calF_\rmd D(\bU) \calF_\rmd^{-1}$ is not real (it is skew-Hermite).
Therefore, the transformed linear system is 
out of the range of application of the COCG/COCR method.
Therefore, it is natural to apply the Bi-CGSTAB method to the transformed linear system. We will observe the performance of the preconditioned Bi-CGSTAB method.
We note that although the coefficient matrix of the transformed linear system is no longer symmetric,
the sum of the first two terms $I + \rmi \Delta t D_\alpha$
remains symmetric and the third term $-\rmi \Delta t \calF_\rmd D(\bU) \calF_\rmd^{-1}$ could be regarded as a perturbation.
It is thus still of interest to investigate what happens if we 
aggressively apply the COCG/COCR method to the transformed linear system \eqref{eq:ls2}
with the preconditioner \eqref{preconditioner}.
In Section~\ref{subsec:peformpre}, we will also see how this approach works.

\section{Numerical experiments}
\label{sec4}

We now test our linearly implicit scheme~\eqref{scheme1}.
First, we check the qualitative behaviour  of the numerical solutions obtained from the linearly implicit scheme~\eqref{scheme1} by comparing the results with those from the fully nonlinear scheme~\eqref{scheme3}.
We also check to what extent $\tilde{\calH}_\rmd (\bU^{(n)})$ remains close to $\tilde{\calH}_\rmd(\bU^{(0)})$.
Next, we discuss the efficiency of the scheme with particular emphasis on linear solvers,
and observe how the preconditioned methods work.

All the computations were performed in a computation environment: 1.6 GHz Intel Core i5, 16GB memory, OS X 10.14.
We use Julia version 1.1.1.

\subsection{Numerical behaviour}

We check the numerical behaviour of the proposed scheme~\eqref{scheme1}
and compare the results with reference solutions which are computed by the nonlinear scheme~\eqref{scheme3}.
The linear system in~\eqref{scheme2} is solved by the conjugate orthogonal conjugate gradient (COCG) method~\cite{vm90} in this subsection
(Algorithm~\ref{algo:pcocg} summarizes the preconditioned COCG method, but it reduces to the COCG method with $M=I$).
As an initial guess of the COCG method, we use the numerical solution at the current time step.
The convergence criteria is set to $10^{-10}$ in terms of the relative $2$-norm residual.
The nonlinear system \eqref{scheme3} is solved by \textsf{nlsolve}\footnote{The function \textsf{nlsolve} is a typical nonlinear solver in Julia. \url{https://pkg.julialang.org/docs/NLsolve/}} with the tolerance $10^{-10}$.

As an example, we set $L=20$, $N=101$ and $\Delta t = 0.02$.
The initial value is set to $u_0 (x) = 2\exp (0.5\rmi x)\sech(\sqrt{2}(x-10))$,
which is a snapshot of a solitary wave solution for the case $\alpha=2$.

Figs.~\ref{fig:a2},~\ref{fig:a16} and~\ref{fig:a12} show the contour of the absolute value of numerical results for several $\alpha$ obtained by the linearly
implicit scheme~\eqref{scheme1} and the nonlinear scheme~\eqref{scheme3}.
It is observed that
the linearly implicit scheme exhibits qualitatively comparable results
to the expensive nonlinear scheme.
We note that if $N$ gets further small the  behaviour is deteriorated as shown in Fig.~\ref{fig:a16_s}.
This figure shows the result for the case $\alpha=1.6$ and $N=61$.
It is observed that the speed of the wave tends to be slower than that of the reference solution.
Let us also check the behaviour in terms of the choice of the time step size in more detail for the case $\alpha = 1.6$.
The results are displayed in Fig.~\ref{fig:globalerror}.
From the left figure, it is observed that the global  error becomes small as the time step sizes $\Delta t$ gets small. Since the scheme is symmetric we expect the second order convergence. From the right figure it seems that the scheme is actually of order two.
Figs.~\ref{fig:inverr_a2},~\ref{fig:inverr_a16} and~\ref{fig:inverr_a12}
show errors of the discrete mass and energy.
For the discrete mass $| \calM_\rmd (\bU^{(n)}) - \calM_\rmd (\bU^{(0)})| $ is plotted and
for the discrete energy $| \calH_\rmd (\bU^{(n+1)},\bU^{(n)}) - \calH_\rmd (\bU^{(1)},\bU^{(0)})| $ is plotted.
Both discrete quantities are well-preserved as expected (note that the tolerance of the linear and nonlinear solvers
are set to $10^{-10}$ and the scheme is computed $1.25\times 10^4$ times until $t=250$).

\begin{figure}[htbp]
\centering
\includegraphics[scale=1]{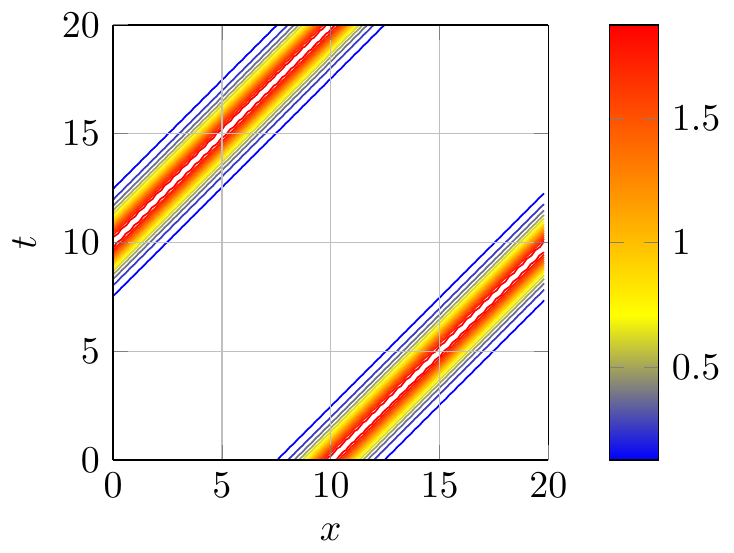}
\includegraphics[scale=1]{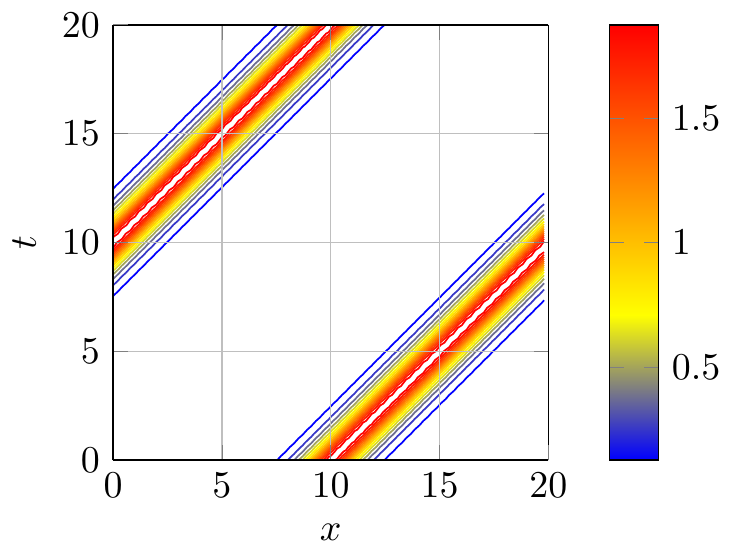}
\caption{Numerical solutions for the case $\alpha = 2$
obtained by (LEFT) the linearly implicit scheme~\eqref{scheme1} and (RIGHT) the nonlinear scheme~\eqref{scheme3}.}
\label{fig:a2}
\end{figure}

\begin{figure}[htbp]
\centering
\includegraphics[scale=1]{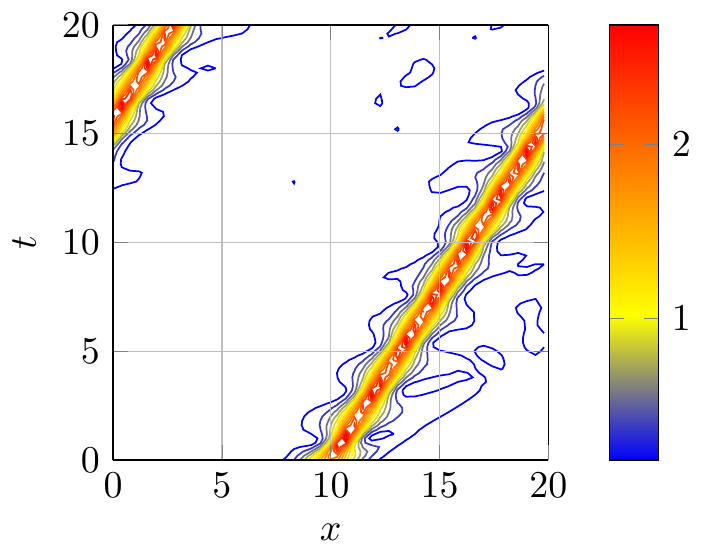}
\includegraphics[scale=1]{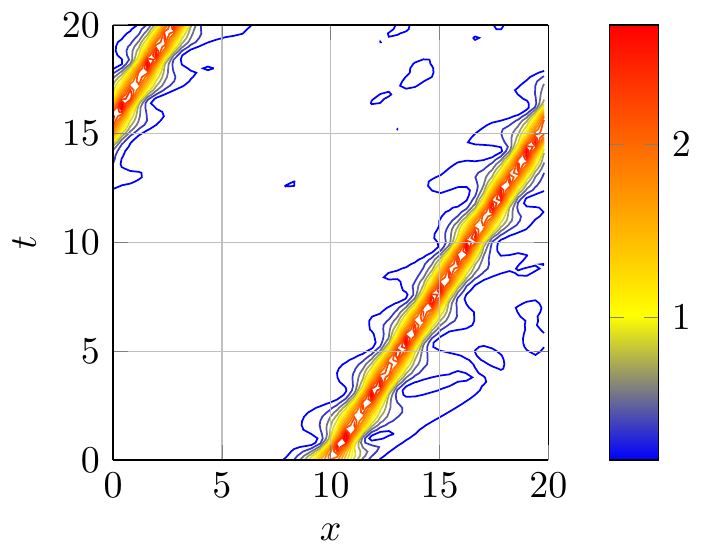}
\caption{Numerical solutions for the case $\alpha = 1.6$
obtained by (LEFT) the linearly implicit scheme~\eqref{scheme1} and (RIGHT) the nonlinear scheme~\eqref{scheme3}.}
\label{fig:a16}
\end{figure}

\begin{figure}[htbp]
\centering
\includegraphics[scale=1]{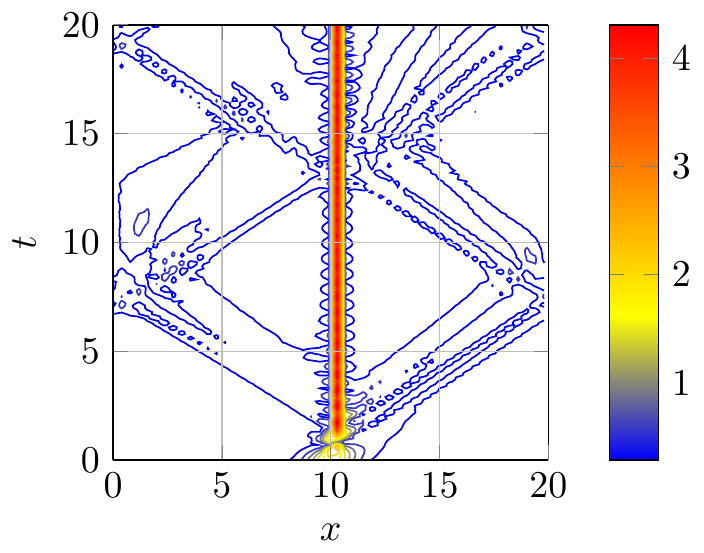}
\includegraphics[scale=1]{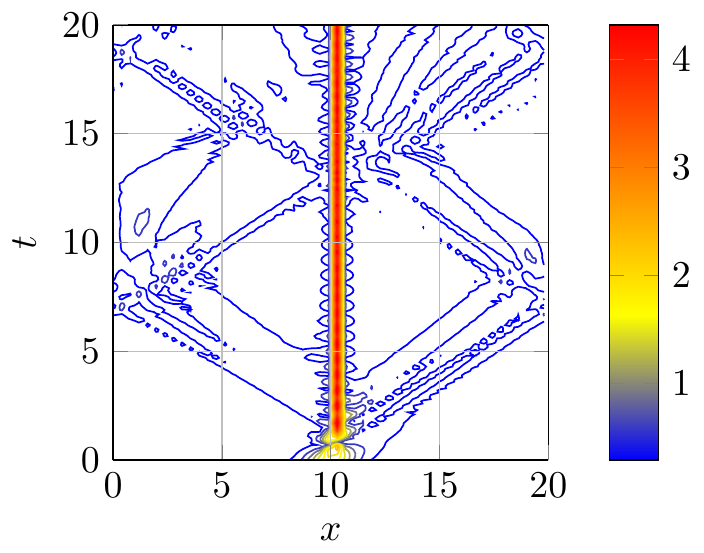}
\caption{Numerical solutions for the case $\alpha = 1.2$
obtained by (LEFT) the linearly implicit scheme~\eqref{scheme1} and (RIGHT) the nonlinear scheme~\eqref{scheme3}.}
\label{fig:a12}
\end{figure}

\begin{figure}[htbp]
\centering
\includegraphics[scale=1]{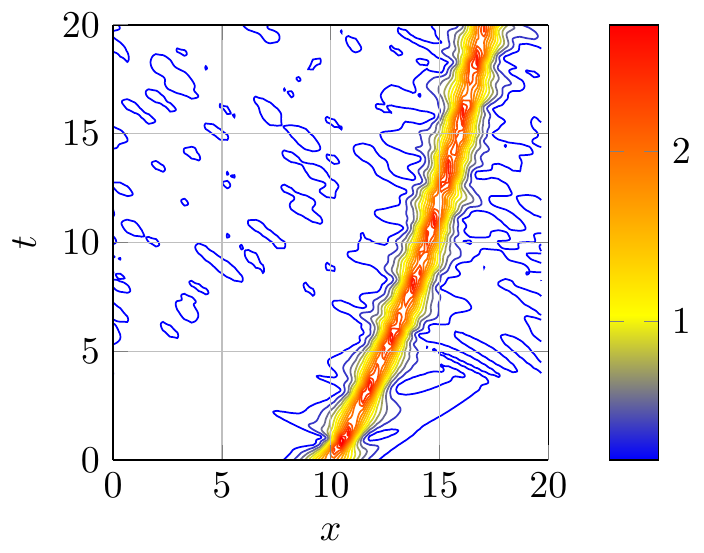}
\caption{Numerical solutions for the case $\alpha = 1.6$
obtained by the linearly implicit scheme~\eqref{scheme1} with $N=61$.}
\label{fig:a16_s}
\end{figure}

\begin{figure}[htbp]
\centering
\includegraphics[scale=1]{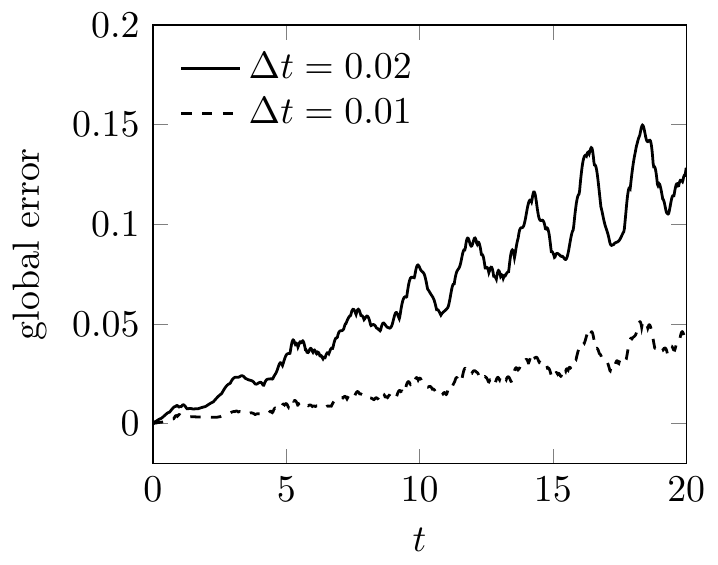}
\includegraphics[scale=1]{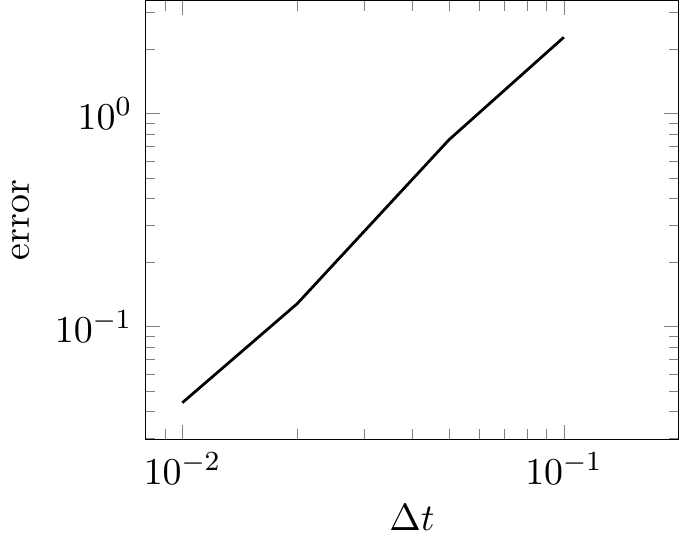}
\caption{
Error behaviour obtained by the linearly implicit scheme~\eqref{scheme1} for the case $\alpha = 1.6$: (LEFT) global error, and
(RIGHT) error at $t=20$.
Errors are measured by $\displaystyle  \max_k | U_k^{(n)} - U_{\text{ref},k}^{(n)}|$,
where the reference solution was generated by the nonlinear scheme~\eqref{scheme3} with $N=303$ and $\Delta t = 0.001$.
}

\label{fig:globalerror}
\end{figure}

\begin{figure}[htbp]
\centering
\includegraphics[scale=1]{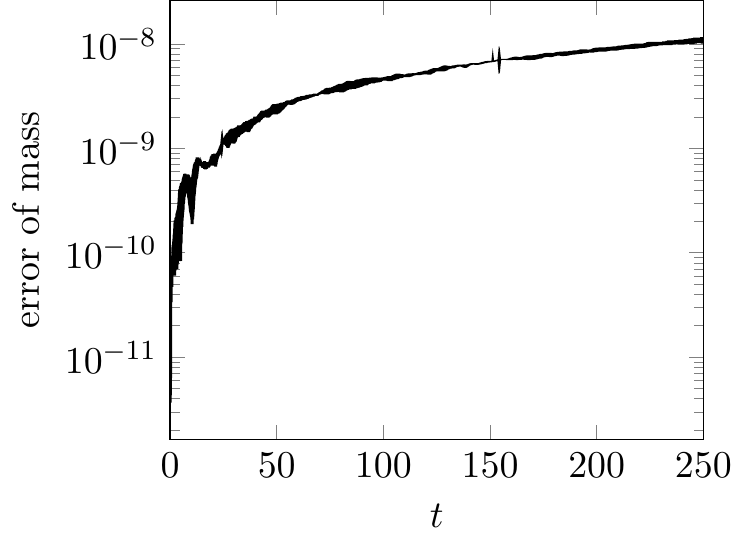}
\includegraphics[scale=1]{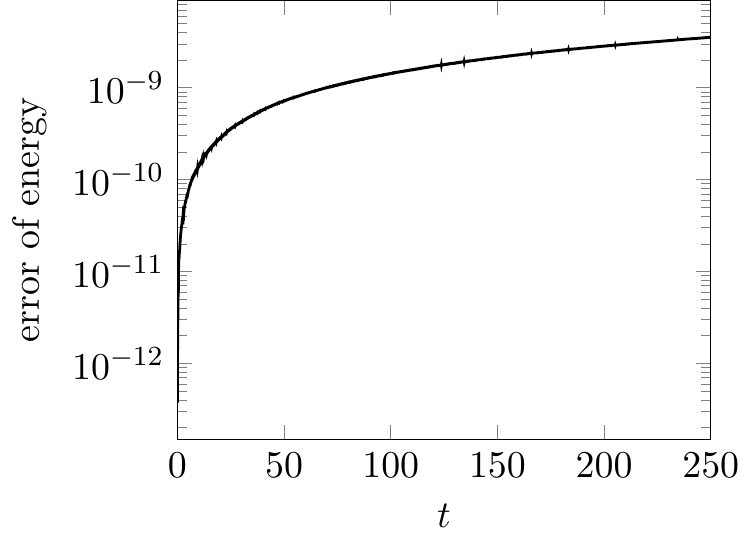}
\caption{Errors of the discrete mass $\calM_\rmd (\bU^{(n)}$ and energy $\calH_\rmd (\bU^{(n+1)},\bU^{(n)})$ obtained by the linearly implicit scheme~\eqref{scheme1}
for the case $\alpha=2.0$. 
}
\label{fig:inverr_a2}
\end{figure}

\begin{figure}[htbp]
\centering
\includegraphics[scale=1]{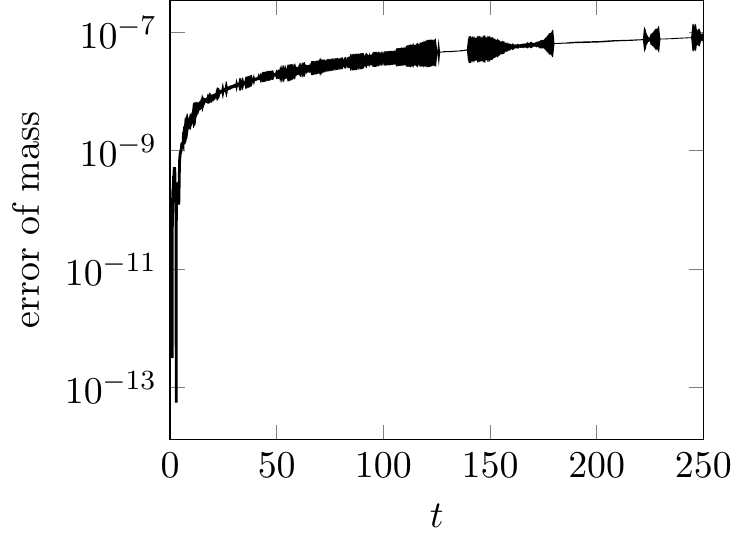}
\includegraphics[scale=1]{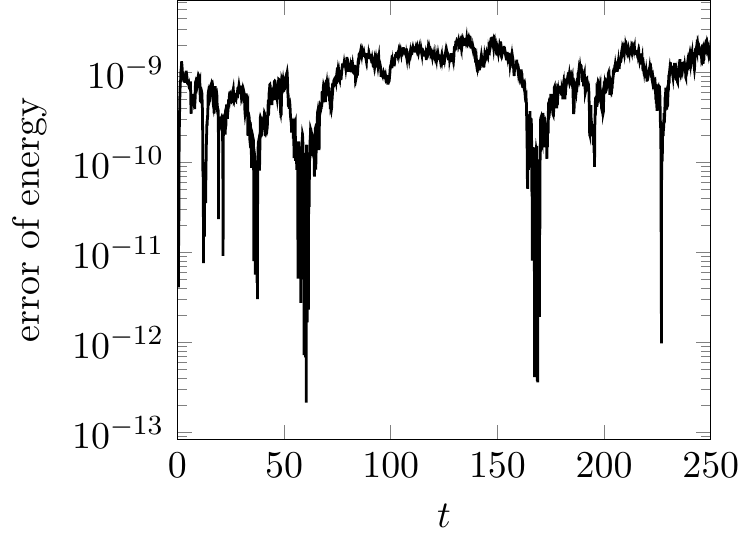}
\caption{Errors of the discrete mass $\calM_\rmd(\bU^{(n)}$ and energy $\calH_\rmd (\bU^{(n+1)},\bU^{(n)})$ obtained by the linearly implicit scheme~\eqref{scheme1}
for the case $\alpha=1.6$. 
}
\label{fig:inverr_a16}
\end{figure}

\begin{figure}[htbp]
\centering
\includegraphics[scale=1]{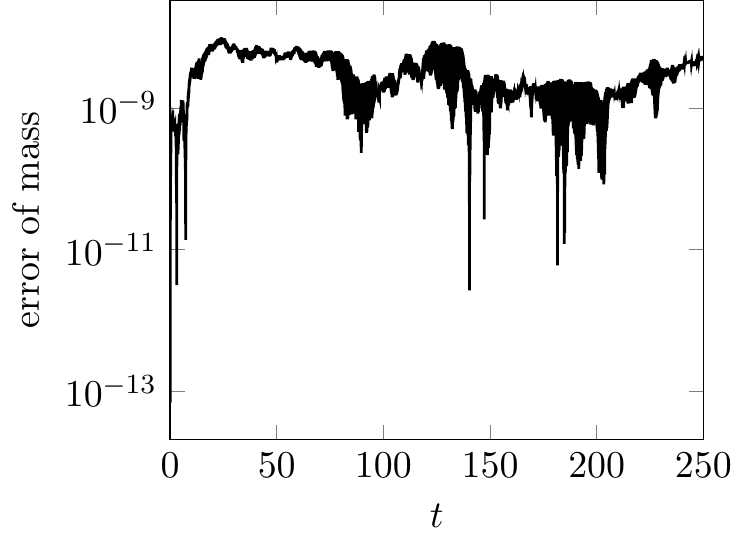}
\includegraphics[scale=1]{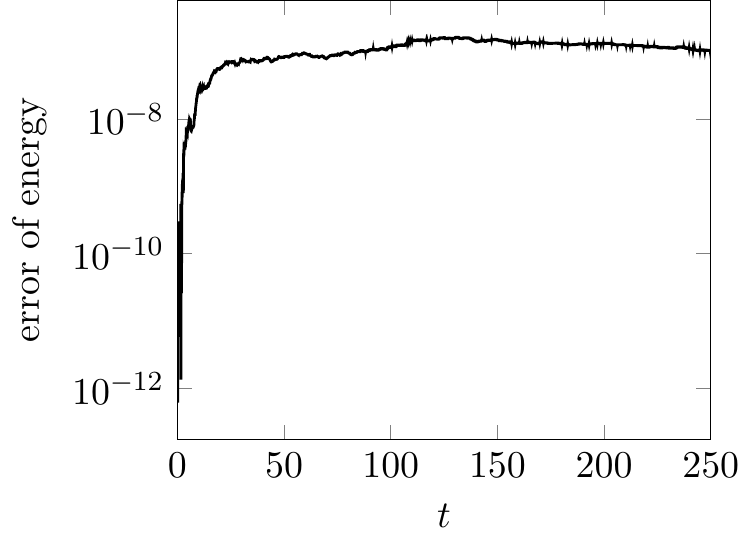}
\caption{Errors of the discrete mass $\calM_\rmd (\bU^{(n)}) $ and energy $\calH_\rmd (\bU^{(n+1)},\bU^{(n)})$ obtained by the linearly implicit scheme~\eqref{scheme1}
for the case $\alpha=1.2$. 
}
\label{fig:inverr_a12}
\end{figure}

\subsection{Preservation of the standard energy}
\label{subsec:energy-pre}

While the discrete mass $\calM_\rmd (\bU^{(n)})$ 
defined on a single time step is preserved,
the discrete energy $\tilde{\calH} _\rmd (\bU^{(n)})$ defined in \eqref{de1} is
not a conserved quantity.
We here investigate to what extent $\tilde{\calH}_\rmd (\bU^{(n)})$ remains close to $\tilde{\calH}_\rmd(\bU^{(0)})$
because
the value $\tilde{\calH}_\rmd (\bU^{(n)})$  could be a good barometer when we consider the long-time stability.
It seems quite challenging to obtain an a priori estimate for the scheme~\eqref{scheme1},
and thus we consider this numerically.
Fig.~\ref{fig:enerr_a16}
shows the results for the case $\alpha = 1.6$.
It is observed that $|\tilde{\calH}_\rmd (\bU^{(n)}) -\tilde{ \calH}_\rmd (\bU^{(0)})| $ is bounded by $10^{-2}$
when $t\leq 300$.
However, when $t$ exceeds $300$, the error becomes large with strong oscillation,
and thus we cannot expect an error bound for $\tilde{\calH}_\rmd (\bU)$ for all $t$.
With other choices of parameters, qualitatively similar behaviour is observed.
These observations indicate that the instability might be caused for a very long-time integration.
This could be a drawback of the proposed linearly implicit scheme
compared with the nonlinear scheme \eqref{scheme3}.
However, we emphasize that since $\tilde{\calH}_\rmd (\bU^{(n)})$
is easily monitored during the time integration, 
we could easily detect a sign of the instability.

\begin{figure}[htbp]
\centering
\includegraphics[scale=1]{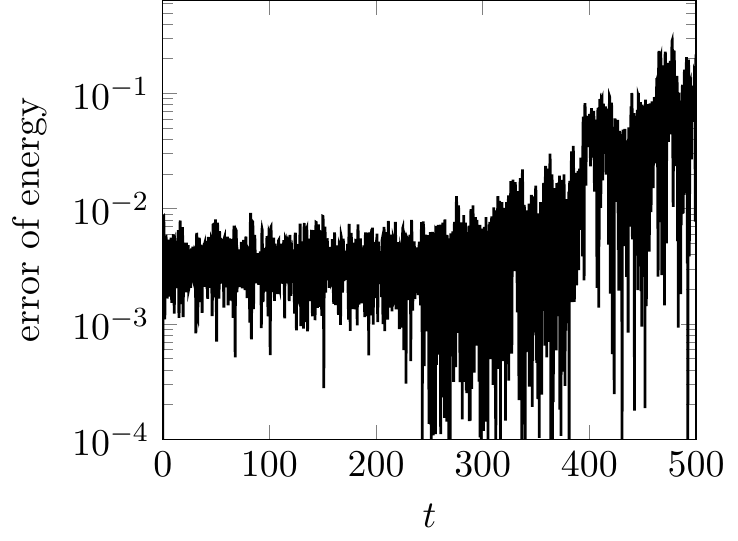}
\caption{Errors of the discrete energy $\tilde{\calH}_\rmd (\bU^{(n)})$ obtained by the linearly implicit scheme \eqref{scheme1}
for the case $\alpha=1.6$.
}
\label{fig:enerr_a16}
\end{figure}

\subsection{Performance of the preconditioning}
\label{subsec:peformpre}

We here discuss the performance of linear solvers
for solving \eqref{scheme2}.

First, we consider solving the original system \eqref{scheme2}
by the COCG method.
Fig.~\ref{fig:numofits1} shows the number of iterations required to the convergence for several settings.
It is observed that more iterations are required for large $\alpha$, $N$, $\Delta t$.
In particular, there is a significant gap when we change $N$ or $\Delta t$.
However, we would also like to emphasize that even in the worst case ($\alpha = 2$, $N = 401$ and $\Delta t = 0.02$),
the result is much better than that by the Bi-CGSTAB method,
which is illustrated in Fig.~\ref{fig:numofits1_BICGSTAB}.
We thus conclude that when we solve the original system \eqref{scheme2} directly, the COCG method seems an appropriate choice. We also note that the COCR method gives 
comparable results to the COCG method.

\begin{figure}[htbp]
\centering
\includegraphics[scale=1]{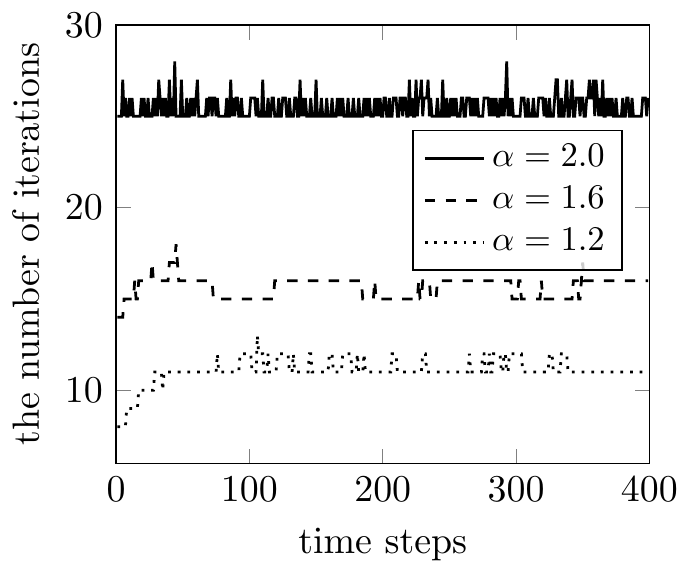}
\includegraphics[scale=1]{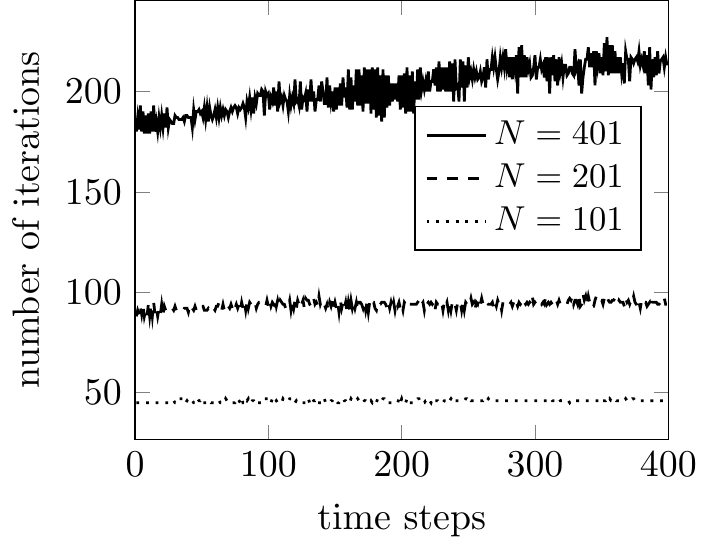}
\includegraphics[scale=1]{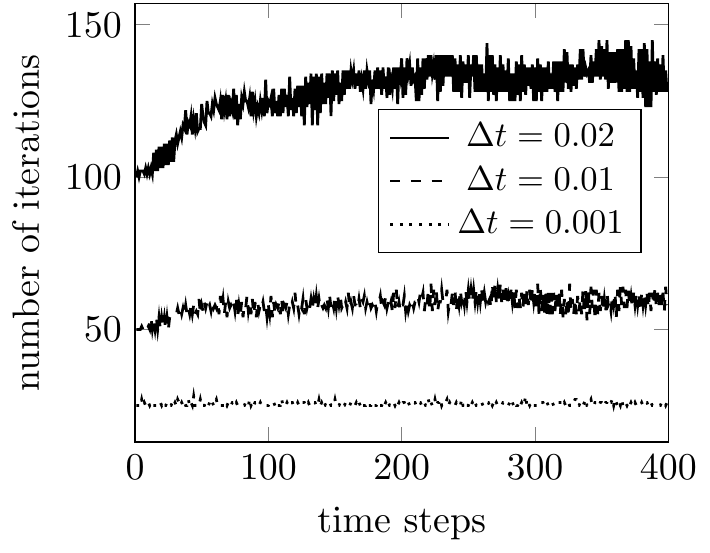}
\caption{The number of iterations for the COCG method to the convergence at each time step. 
(LEFT) $N=101$ and $\Delta t = 0.02$ are fixed, (RIGHT) $\alpha=2$ and $\Delta t = 0.02$ are fixed,
(BOTTOM) $\alpha=2$ and $N=401$ are fixed.
}
\label{fig:numofits1}
\end{figure}

\begin{figure}[htbp]
\centering
\includegraphics[scale=1]{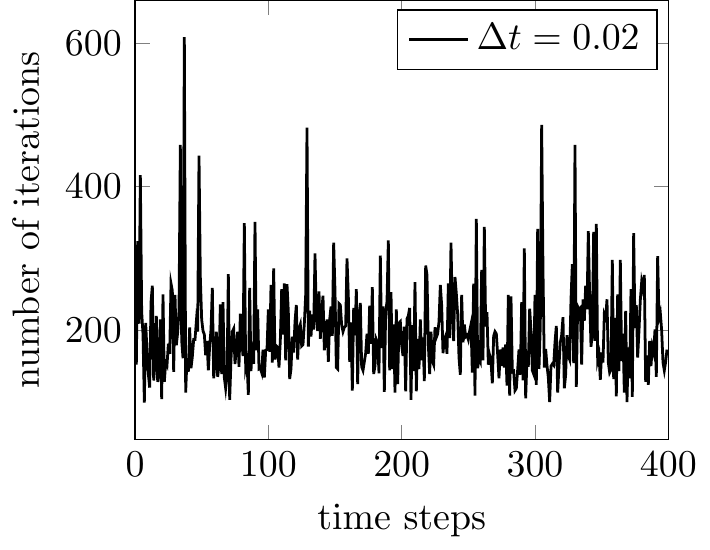}
\caption{The number of iterations for the Bi-CGSTAB method applied to \eqref{scheme2}.
The parameters are set to 
$\alpha=2$, $N=401$ and $\Delta t = 0.02$.
}
\label{fig:numofits1_BICGSTAB}
\end{figure}

Although the COCG method is preferred for solving \eqref{scheme2},
it is hoped that the convergence behaviour is improved.
Thus, we next discuss how the variable transformation and preconditioner proposed in Section~\ref{subsec:pcond} work.
In the following numerical experiments, we consider the case $\alpha=2$.
Table~\ref{table1} shows the maximum, minimum and average number of iterations of the preconditioned Bi-CGSTAB method for several $N$, where $T = 8$ with the time step size $\Delta t=0.02$
and the initial value is set to $u_0 (x) = 2\exp (0.5\rmi x)\sech(\sqrt{2}(x-10))$.
The convergence behaviour is outstandingly improved
for the case $N=401$ compared with Fig.~\ref{fig:numofits1_BICGSTAB},
and furthermore,
the results are notable in that for all cases the preconditioned Bi-CGSTAB method requires only three iterations.
In this problem setting, the CPU time is shown in Table~\ref{table1-a}.
The computation time seems to be almost proportional to $N\log N$ (in the sense that it is a bit worse than $N$, but much better than $N^2$).
Let us change the time step size to $\Delta t = 0.2$.
The results are shown in Table~\ref{table2}.
By comparing Table~\ref{table2} with Table~\ref{table1}, we observe that the convergence behaviour depends on $\Delta t$, but the results still remain outstanding.
Let us also change the initial value. The results are shown in Table~\ref{table3}, which indicate that the dependency on the shape of solutions is subtle.

\begin{table}[htbp]
\caption{The maximum, minimum and average number of iterations of the preconditioned Bi-CGSTAB method:
the time step size is set to $\Delta t = 0.02$, and the initial value $u_0 (x) = 2\exp (0.5\rmi x)\sech(\sqrt{2}(x-10))$.} 
\label{table1}
\centering
\begin{tabular}{cccc}
\hline 
$N$ & $401$ & $1001$ & $4001$ \\ 
\hline 
maximum & $3$ & $3$ & $3$ \\ 
minimum & $3$ & $3$ & $3$ \\ 
average & $3$ & $3$ & $3$ \\ 
\hline 
\end{tabular} 
\end{table}

\begin{table}[htbp]
\caption{Average CPU time of $10$ simulations at $T=8$ (the cost for obtaining $\bU^{(1)}$ by the nonlinear scheme~\eqref{scheme3} is excluded):
the time step size is set to $\Delta t = 0.02$, and the initial value $u_0 (x) = 2\exp (0.5\rmi x)\sech(\sqrt{2}(x-10))$.} 
\label{table1-a}
\centering
\begin{tabular}{ccccc}
\hline 
$N$ & $1001$ & $2001$ & $4001$ & $8001$ \\ 
\hline 
CPU time & $0.794$ & 
$1.524$ & $5.615$ & $8.223$ \\ 
\hline 
\end{tabular} 
\end{table}

\begin{table}[htbp]
\caption{The maximum, minimum and average number of iterations of the preconditioned Bi-CGSTAB method:
the time step size is set to $\Delta t = 0.2$, and the initial value $u_0 (x) = 2\exp (0.5\rmi x)\sech(\sqrt{2}(x-10))$.} 
\label{table2}
\centering
\begin{tabular}{cccc}
\hline 
$N$ & $401$ & $1001$ & $4001$ \\ 
\hline 
maximum & $6$ & $6$ & $6$ \\ 
minimum & $5$ & $5$ & $5$ \\ 
average & $5.020$ & $5.020$ & $5.020$ \\ 
\hline 
\end{tabular} 
\end{table}

\begin{table}[htbp]
\caption{The maximum, minimum and average number of iterations of the preconditioned Bi-CGSTAB method:
the time step size is set to $\Delta t = 0.02$, and the initial value $u_0 (x) = 2\exp (0.5\rmi x)\sech(x-10)$.} 
\label{table3}
\centering
\begin{tabular}{cccc}
\hline 
$N$ & $401$ & $1001$ & $4001$ \\ 
\hline 
maximum & $4$ & $4$ & $4$ \\ 
minimum & $3$ & $3$ & $3$ \\ 
average & $3.286$ & $3.291$ & $3.323$ \\ 
\hline 
\end{tabular} 
\end{table}

As discussed in Section~\ref{subsec:pcond},
it is also of interest to investigate the behaviour when
the COCG method is aggressively applied to the transformed system~\eqref{eq:ls2}
with the preconditioner \eqref{preconditioner},
since the coefficient matrix in~\eqref{eq:ls2}
can be seen as a complex symmetric matrix plus a perturbation.
Fig.~\ref{fig:numofits2} shows the results.
From the left figure, 
it is observed that, when $\Delta t = 0.01$ and $0.02$, 
the preconditioned COCG method actually work and the results are significantly improved
compared with the bottom figure in Fig.~\ref{fig:numofits1}.

Unfortunately, however, if we use a larger time step size $\Delta t = 0.05$, the preconditioned COCG method requires 50 iterations
at the $7$th time step as shown in the right figure of Fig.~\ref{fig:numofits2}, and the iteration does not converge within $1,000$ iterations
at the $8$th time step.
This observation indicates that with the step size $\Delta t = 0.05$ the influence of the perturbation term
is not negligible.

Let us change the initial condition
to $u_0 (x) = 2\exp (0.5\rmi x)\sech(4(x-10))$.
This function has a steeper slope.
The results are displayed in Fig.~\ref{fig:numofits3}.
It is observed that even if a much larger time step size $\Delta t = 0.2$ is employed,
the preconditioned COCG method works fine.
Conversely, a more gradual initial condition $u_0 (x) = 2\exp (0.5\rmi x)\sech(x-10)$
with the time step size $\Delta t = 0.02$
was also considered as our preliminary experiments, and it was observed
that in this case the preconditioned COCG method did not converge at $7$th time step.
These observations indicate that the convergence of the preconditioned COCG method strongly depends on $\Delta t$ and the shape of the solution (in other words, the influence of $D(\bU^{(n)})$). 
They make the effect of the perturbation term in \eqref{eq:ls2}
significant.

For the problem considered in this paper, it is highly recommended to use the Bi-CGSTAB method with the proposed variable transformation and preconditioner, but it is of interest to understand the behaviour of the preconditioned COCG method, which will be investigated in our future work.

\begin{figure}[htbp]
\centering
\includegraphics[scale=1]{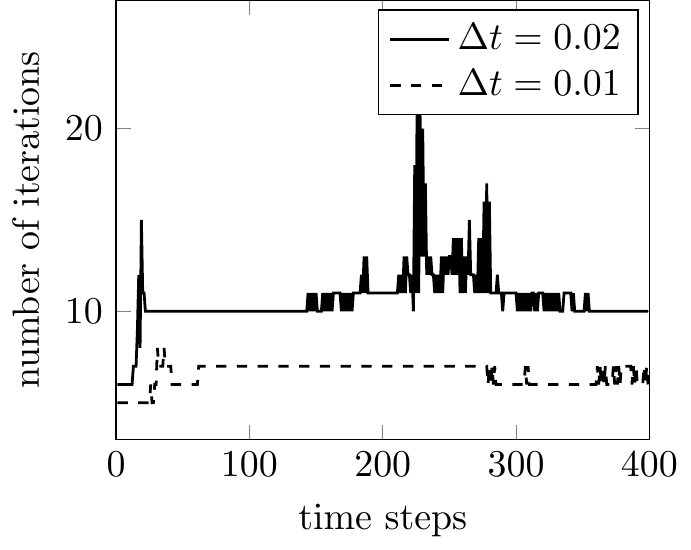}
\includegraphics[scale=1]{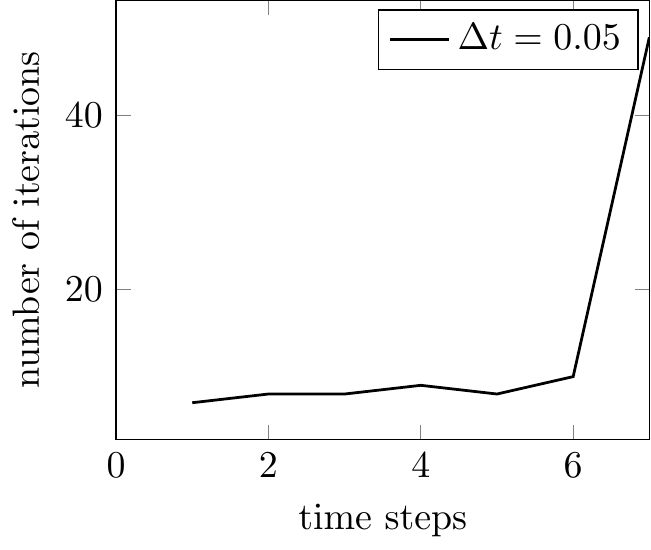}
\caption{The number of iterations for the preconditioned COCG method 
applied to \eqref{eq:ls2} with the matrix $M= I + \rmi \Delta t D_\alpha$.
The initial condition is set to $u_0 (x) = 2\exp (0.5\rmi x)\sech(\sqrt{2}(x-10))$.
The parameters are set to $\alpha=2$ and $N=401$.
(LEFT) $\Delta t = 0.01,0.02$, (RIGHT) $\Delta t  =0.05$.
}

\label{fig:numofits2}
\end{figure}

\begin{figure}[htbp]
\centering
\includegraphics[scale=1]{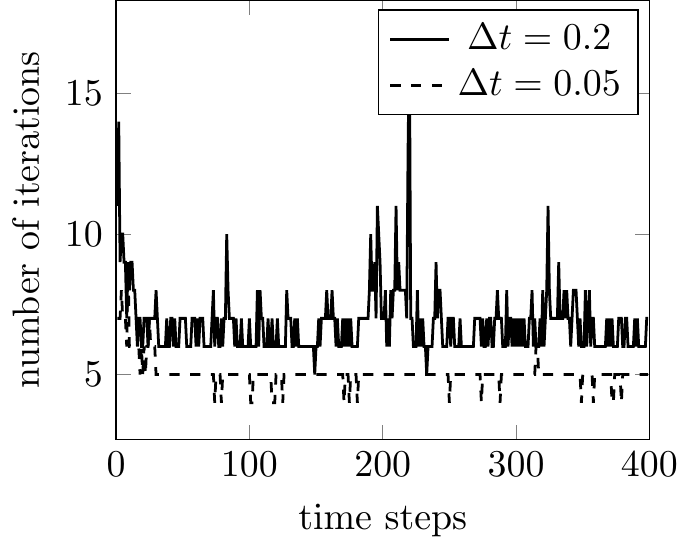}
\caption{The number of iterations for the preconditioned COCG method 
applied to \eqref{eq:ls2} with the matrix $M= I + \rmi \Delta t D_\alpha$.
The initial condition is set to $u_0 (x) = 2\exp (0.5\rmi x)\sech(4(x-10))$.
The parameters are set to $\alpha=2$ and $N=401$.
}
\label{fig:numofits3}
\end{figure}

\section{Concluding remarks}
\label{sec5}

In this paper,
we proposed the linearly implicit scheme~\eqref{scheme1} for the FNLS equation
preserving two invariants: mass  and energy.
The scheme exhibited qualitatively comparable results to the expensive  nonlinear scheme~\eqref{scheme3}.
The preconditioning issues were also discussed:
the preconditioned Bi-CGSTAB method is a preferable choice.

We note several directions for future work.

\begin{itemize}
\item It is hoped that the proposed scheme is used to investigate more challenging problems such as
multi-dimensional problems. When we consider the multi-dimensional problems, the computational complexity
and preconditioning issues become increasingly important, and thus the discussion on the preconditioning considered
for the one-dimensional problem would be helpful.
\item The linearly implicit scheme does not preserve $\tilde{\calH}_\rmd (\bU)$, which is defined on a numerical solution
of a single time step.
The results presented in Fig.~\ref{fig:enerr_a16} should be theoretically investigated in more detail.
Note that structure-preserving linearly implicit schemes preserving a certain quantity have also been
proposed for other partial differential equations (see, e.g.~\cite{do11,mm12,mmf11}),
and similar behaviour might also  have to be reconsidered as well.
\item The discussion in Section~\ref{subsec:peformpre}
indicates that the COCG/COCR method is applicable to
complex but non-symmetric matrices if the non-symmetric term can be regarded as a perturbation in some sense.
This behaviour will be investigated theoretically in more detail.
\end{itemize}


\bibliographystyle{AIMS}
\bibliography{references}
\end{document}